\documentclass{amsart}

\usepackage{amssymb}
\usepackage{amsfonts}
\usepackage{amsmath}

\newtheorem{theorem}{Theorem}[section]

\newtheorem{lemma}[theorem]{Lemma}

\newtheorem{definition}{Definition}[section]

\makeatletter
\@namedef{subjclassname@2020}{%
  \textup{2020} Mathematics Subject Classification}
\makeatother

\begin{document}

\title[A variation of continuity in $n$-normed spaces]{A variation of continuity in $n$-normed spaces}

\author[S. Ersan]
{Sibel Ersan \\ \small{ Maltepe University \\ Turkey\\ sibelersan@maltepe.edu.tr}}
\address{Sibel Ersan \\ Faculty of Engineering and Natural Sciences \\ Maltepe University \\ Istanbul, Turkey}
\email{sibelersan@maltepe.edu.tr}

\begin{abstract}
The s-th forward difference sequence that tends to zero, inspired by the consecutive terms of a sequence approaching zero, is examined in this study. Functions that take sequences satisfying this condition to sequences satisfying the same condition are called s-ward continuous. Inclusion theorems that are related to this kind of uniform continuity and continuity are also considered. Additionally, the concept of $s$-ward compactness of a subset of $X$ via $s$-quasi-Cauchy sequences are investigated. One finds out that the uniform limit of any sequence of $s$-ward continuous function is $s$-ward continuous and the set of $s$-ward continuous functions is a closed subset of the set of continuous functions.
\end{abstract}



\maketitle


\section{Introduction and Preliminaries}
\label{Sec:1}
Although some evaluations was first made about the axioms of an abstract n-dimensional metric, the main developments regarding the definition of the 2-metric, 2-normed spaces and their topological properties were described by G\"ahler \cite{gohler} then the results of these concepts were extended to the most generalized case $n$-metric and $n$-normed spaces where $n$ is a natural number by G\"ahler\cite{gahler1}. Shortly after the concept of $n$-normed space is introduced, the concept of $2$-inner product space is also defined in \cite{gahler4}. Afterwards many authors have done lots of impressive improvements in $n$-normed spaces or in $2$-inner product spaces (\cite{misiak,Kim,Malceski,gunawan, gurdal,dutta,gurdalsari}).

Firstly we recall the notion of $n$-normed space:
\begin{definition} An $n$-norm on a real vector space $X$ of dimension $d$,  where $2\leq n\leq d$ is a real valued function $\|.,...,.\|$ on $X^n$ which satisfies the properties:
\begin{enumerate}
\item	$\|\zeta_1,\zeta_2,...,\zeta_n\|=0$ if and only if $\zeta_1,\zeta_2,...,\zeta_n$ are linearly dependent,
\item	$\|\zeta_1,\zeta_2,...,\zeta_n\|=||\zeta_{k_1},...,\zeta_{k_n}||$ for every permutation $(k_1,...,k_n)$ of $(1,...,n)$,
\item	$\|\zeta_1,\zeta_2,...,\delta \zeta_n\|=|\delta|\|\zeta_1,\zeta_2,...,\zeta_n\|$ for any real number $\delta$,
\item	$\|m+n,\zeta_1,...,\zeta_{n-1}\|\leq\|m,\zeta_1,...,\zeta_{n-1}\|+\|n,\zeta_1,...,\zeta_{n-1}\|$.
\end{enumerate}
A set $X$ is an $n$-normed space with an n-norm $\|.,...,.\|$.
\end{definition}

In \cite{hgunawan},
\begin{equation}
\|\zeta_1,...,\zeta_n\|_p=\left[\frac{1}{n!}\sum_{t_1}...\sum_{t_n} \det(\zeta_{it_k})^p\right]^{1/p}.
\end{equation}
is given as an example of an $n$-norm on $l^p\times...\times l^p$ space for $1\leq p<\infty$. Also if $p=\infty$, the $n$-norm on $l^\infty\times...\times l^\infty$ is given as  \cite{malceski}
\begin{equation}
\|\zeta_1,...,\zeta_n\|_{\infty}=\sup_{t_{1}}...\sup_{t_{n}}\det (x_{it_k}).
\end{equation}

\begin{definition}A sequence $(x_{k})$ converges to an $\zeta\in X$ in an $n$-normed space $X$ if for each $\epsilon>0$, there exists a positive integer $\tilde{k}$ such that for every $k\geq \tilde{k}$
\begin{equation}\label{1}
\|x_{k}-\zeta,\mu_1,\ldots,\mu_{n-1}\|<\epsilon, \ \ \ \ \forall \mu_1,\ldots,\mu_{n-1}\in X.
\end{equation}
\end{definition}
\begin{definition}
A sequence $(x_{k})$ is a Cauchy sequence if for each $\epsilon>0$, there exists a positive integer $t_0$ such that for every $k,m\geq t_0$
\begin{equation}\label{1}
\|x_{k}-x_{m},\mu_1,\ldots,\mu_{n-1}\|<\epsilon, \ \ \ \ \forall \mu_1,\ldots,\mu_{n-1}\in X.
\end{equation}
\end{definition}
If each Cauchy sequence in $X$ converges to an element of $X$, we call $X$ is complete and if $X$ is complete, then it is called an $n$-Banach space.

In recent times, the notion of quasi-Cauchy sequences is given in \cite{burton}. The distance between consecutive terms of a sequence tending to zero is expressed by Burton and Coleman with the idea of quasi-Cauchy sequence. Then using this idea, different types of continuities were defined for real functions in \cite{CakalliForwardcontinuity, CakalliVariationsonQuasi-CauchySequences} as ward continuity, statistically ward continuity, lacunary ward continuity and etc. They were also studied in $2$-normed space in \cite{ersan,CakalliersanStronglylacunarywardcontinuityin2-normedspaces,lacunary}.

The aim of this research is to give a generalization of the notions of a quasi-Cauchy sequence and ward continuity of a function to the notions of an $s$-quasi-Cauchy sequence and $s$-ward continuity of a function in an $n$-normed space for any fixed positive integer $s$. Also interesting theorems related to ordinary continuity, uniform continuity, compactness and $s$-ward continuity are proved. This paper contains not only an extension of results of \cite{ersan} to an $n$-normed space, but also includes new results in $2$-normed spaces as a special case for $n=2$.

\section{Main results}
\label{Sec:2}
In this paper $X$, $\mathbb{R}$ and $s$ will denote a first countable $n$-normed space with an $n$-norm $\|.,...,.\|$, the set of all real numbers and a fixed positive integer, respectively. Now we give the notion of $s$-quasi Cauchyness of a sequence in $X$:
\begin{definition}
A sequence $(x_{k})$ of points in $X$ is said to be $s$-quasi-Cauchy if for all $\mu_1,\mu_2,...,\mu_{n-1}\in X$ it satisfies
\begin{equation}
\lim_{k\rightarrow \infty} ||\Delta_{s} x_{k},\mu_1,\mu_2,...,\mu_{n-1}||=0
\end{equation}
where $\Delta_{s} x_{k}=x_{k+s}-x_{k}$ for each positive integer $k$.
\end{definition}
If one chooses $s=1$, the sequences returns to the ordinary quasi-Cauchy sequences and also using the equality

$$x_{k+s}-x_{k}=x_{k+s}-x_{k+s-1}+x_{k+s-1}-x_{k+s-2}...-x_{k+2}+x_{k+2}-x_{k+1}+x_{k+1}-x_{k},$$
we see that any quasi-Cauchy sequence is $s$-quasi-Cauchy, however the converse is not true.

Any Cauchy sequence is $s$-quasi-Cauchy, so is any convergent sequence. A sequence of partial sums of a convergent series is $s$-quasi-Cauchy. One notes that the set $\Delta_s(X)$, the set of $s$-quasi-Cauchy sequences in $X$, is a vector space. If $(x_{k}),(y_{k})$ are $s$-quasi-Cauchy sequences in $X$ so
\begin{eqnarray}
\lim_{k\rightarrow \infty} ||\Delta_{s} x_{k},\mu_1,\mu_2,...,\mu_{n-1}||=0 \ \ \textrm{and}\\
\lim_{k\rightarrow \infty} ||\Delta_{s} y_{k},\mu_1,\mu_2,...,\mu_{n-1}||=0.
\end{eqnarray}
for all $\mu_1,\mu_2,...,\mu_{n-1}\in X$. Therefore
\begin{eqnarray}
\lim_{k\rightarrow \infty} ||\Delta_{s} (x_{k}+y_{k}),\mu_1,\mu_2,...,\mu_{n-1}||\leq \lim_{k\rightarrow \infty} ||\Delta_{s} x_{k},\mu_1,\mu_2,...,\mu_{n-1}||\nonumber\\+\lim_{k\rightarrow \infty} ||\Delta_{s}y_{k},\mu_1,\mu_2,...,\mu_{n-1}||=0.
\end{eqnarray}
So the sum of two $s$-quasi-Cauchy sequence is again $s$-quasi-Cauchy, it is clear that $(ax_k)$ is an $s$-quasi-Cauchy sequence in $X$ for any constant $a\in \mathbb{R}$.
\begin{definition}
A subset $A$ of $X$ is called $s$-ward compact if any sequence in the set $A$ has an $s$-quasi-Cauchy subsequence.
\end{definition}
If a set $A$ is an $s$-ward compact subset of $X$, then any subset of $A$ is $s$-ward compact. Moreover any ward compact subset of $X$ is $s$-ward compact. Union of finite number of $s$-ward compact subset of $X$ is $s$ ward compact. Any sequentially compact subset of $X$ is $s$-ward compact.

For each real number $\alpha>0$, an $\alpha$-ball with center $a$ in $X$ is defined as
\begin{equation}\label{u}
  B_{\alpha}(a,x_1,...,x_{n-1})=\{x\in X:||a-x,x_1-x,...,x_{n-1}-x||<\alpha\}
\end{equation}
for $x_1,...,x_{n-1}\in X$. The family of all sets $W_{i}(a)=B_{\alpha_i}(a,x_{i_1},...,x_{i_{(n-1)}})$ where $i=1,2,..$ is an open basis in $a$.
Let $\beta_{n-1}$ be the collection of linearly independent sets $B$ with $n-1$ elements. For $B\in \beta_{n-1}$, the mapping $$p_{B}(x)=||x,x_1,...,x_{n-1}||,\ \textrm{for} \ x\in X, \ \ x_1,...,x_{n-1}\in B$$ defines a seminorm on $X$ and the collection $\{p_{B}:B\in \beta_{n-1}\}$ of seminorms makes $X$ a locally convex topological vector space. For each $x\in X$, different from zero, there exists $x_1,...,x_{n-1}\in B$ such that $x,x_1,x_2,...,x_{n-1}$ are linearly independent so $||x,x_1,...,x_{n-1}||\neq 0$, which makes $X$ a Hausdorff space.  A neighborhood of origin for this topology is in a form of a finite intersection $$\bigcap_{i=1}^n \{ x\in X:||x,x_{i_1}-x,...,x_{i_{(n-1)}}-x||<\epsilon\}$$ where $\epsilon>0$.

Now the following theorem characterizes totally boundedness not only valid for $n$-normed spaces but also valid for the $2$-normed spaces. It extends the results for quasi-Cauchy sequences given in \cite{ersan} for $2$-normed valued sequences to $n$-normed valued $s$-quasi-Cauchy sequences in which $s=1$ gives earlier results given for $2$-normed spaces. It should be noted that Theorem 3 in \cite{cakallistatistical} can not be obtained just by putting $n=1$ in the $n$-normed space to get in a normed space, which is awkward, whereas the following theorem is interesting as a point of studying a new space.
\begin{lemma}\label{l} If a subset of $X$ is totally bounded then every sequence in $A$ contains an $s$-quasi-Cauchy subsequence.
\end{lemma}
\begin{proof}
Let $A$ be totally bounded. Let $(x_n)$ be any sequence in $A$. Since $A$ is covered by a finite number of balls of $X$ of diameter less than 1.
One of these sets, which we denote by $A_1$, must contain $x_n$ for infinitely many values of $n$. Choose a positive integer $n_1$ such that $x_{n_1}\in A_1$. Since $A_1$ is totally bounded, it is covered by a finite number of balls of diameter less than 1/2. One of these balls, which we denote by $A_2$, must contain $x_n$ for infinitely many $n$. Let $n_2$ be a positive integer such that $n_2>n_1$ and $x_{n_2}\in A_2$. Since $A_2\subset A_1$, it follows that $x_{n_2}\in A_1$. Continuing in this way, a ball $A_k$ of $A_{k-1}$ of diameter less than $1/k$ and a term $x_{n_k}\in A_k$ of the sequence $(x_n)$, where $n_k>n_{k-1}$ for any positive integer $k$. Since $x_{n_k},x_{n_{k+1}},...,x_{n_{k+s}},...$ lie in $A_k$ and diameter $(A_k)$ less than $1/k$, then $x_{n_k}$ is an $s$-quasi-Cauchy subsequence of $(x_n)$.
\end{proof}
\begin{theorem} A subset of $X$ is totally bounded if and only if it is $s$-ward compact for any positive integer $s$.
\end{theorem}
\begin{proof} If $A$ is totally bounded, then every sequence of $A$ has an $s$-quasi-Cauchy subsequence by Lemma \ref{l}. So the set $A$ is $s$-ward compact for any fixed positive integer $s$. For the converse, think $A$ is not a totally bounded set. Choose any $x_1\in A$ and $\alpha>0$. Since $A$ is not totally bounded, the neighborhood of a point $x_1$ in $A$ which is defined by $B_{\alpha}(x_1,\mu^1_1,...,\mu^1_{n-1})=\{y\in A; ||x_1-y,\mu^1_1-y,...,\mu^1_{n-1}-y||<\alpha\}$ is not equal to $A$. There is an $x_2\in A$ such that $x_2\notin B_{\alpha}(x_1,\mu^1_1,...,\mu^1_{n-1})$, that is, $||x_1-x_2,\mu^1_1-x_2,...,\mu^1_{n-1}-x_2||\geq \alpha$. Since $A$ is not totally bounded $B_{\alpha}(x_1,\mu^1_1,...,\mu^1_{n-1})\cup B_{\alpha}(x_2,\mu^2_1,...,\mu^2_{n-1})\neq A$ where $B_{\alpha}(x_2,\mu^2_1,...,\mu^2_{n-1})$ is the neighborhood of a point $x_2$ in $A$. Continuining the procedure, a sequence $(x_k)$ of points in $A$ can be obtained as $x_{k+s} \not\in \bigcup_{i=1}^{k+s-1} B_{\alpha}(x_i,\mu^i_1,...,\mu^i_{n-1})$. So $||x_{k+s}-x_k,\mu^i_1-x_k,...,\mu^i_{n-1}-x_k||\geq\alpha$ and all nonzero $\mu^i_1,...,\mu^i_{n-1}\in A$ where $i=1,...,k+s-1.$ So the sequence $(x_k)$ has not any $s$-quasi-Cauchy subsequence. Therefore $A$ is not $s$-ward compact.
\end{proof}
\begin{definition}A function $f$ is called $s$-ward continuous on a subset $A$ of $X$ if
\begin{equation}
lim_{k\rightarrow\infty}||\Delta_s x_{k},\mu_1,\mu_2,...,\mu_{n-1}||=0
\end{equation}
is satisfied for all $\mu_1,\mu_2,...,\mu_{n-1} \in X$, then
\begin{equation}
lim_{k\rightarrow\infty}||\Delta_s f (x_{k}),f(\mu_1),f(\mu_2),...,f(\mu_{n-1})||=0.
\end{equation}
\end{definition}
In the following we give that any $s$-ward continuous function is continuous.
\begin{theorem}\label{s} Any $s$-ward continuous function on a subset $A$ of $X$ is continuous on $A$.
\end{theorem}
\begin{proof}Let the function $f$ be $s$-ward continuous on $A\subset X$ and any sequence $(x_{k})$ in $A$ converge to $\zeta$, that is
\begin{equation}
lim_{k\rightarrow\infty}||x_{k}-\zeta,\mu_1,\mu_2,...,\mu_{n-1}||=0
\end{equation}
for all $\mu_1,\mu_2,...,\mu_{n-1}\in X$. Let us write a new sequence using some terms of the sequence $(x_k)$ as
$$(t_{m})=(x_{1},...,x_{1},\zeta,...,\zeta,x_{2},...,x_{2},\zeta,...,\zeta,...,x_{n},...,x_{n},\zeta,...,\zeta,...).$$ where same terms repeated s-times. Every convergent sequences is Cauchy and moreover any Cauchy sequence is $s$-quasi-Cauchy, then it follows that
\begin{eqnarray*}
lim_{m\rightarrow\infty}||\Delta_s t_{m},\mu_1,\mu_2,...,\mu_{n-1}||=lim_{m\rightarrow\infty}||t_{m+s}-t_m,\mu_1,\mu_2,...,\mu_{n-1}||=0
\end{eqnarray*}
in which either $$lim_{m\rightarrow\infty}||t_{m+s}-\zeta,\mu_1,\mu_2,...,\mu_{n-1}||=0$$ or $$lim_{m\rightarrow\infty}||\zeta-t_m,\mu_1,\mu_2,...,\mu_{n-1}||=0$$ for every $\mu_1,\mu_2,...,\mu_{n-1}$.
This result implies $(t_{m})$ is an $s$-quasi Cauchy sequence. Since the function $f$ is assumed to be $s$-ward continuous, using this assumption we get
\begin{eqnarray}
lim_{m\rightarrow\infty}||\Delta_s f(t_{m}),f(\mu_1),f(\mu_2),...,f(\mu_{n-1})||\nonumber\\
=lim_{m\rightarrow\infty}||f(t_{m+s})-f(t_{m}),f(\mu_1),f(\mu_2),...,f(\mu_{n-1})||=0
\end{eqnarray}
in which either $$lim_{m\rightarrow\infty}||f(t_{m+s})-f(\zeta),f(\mu_1),f(\mu_2),...,f(\mu_{n-1})||=0$$ or $$lim_{m\rightarrow\infty}||f(\zeta)-f(t_{m}),f(\mu_1),f(\mu_2),...,f(\mu_{n-1})||=0.$$ So $(f(x_{k}))$ converges to $f(\zeta)$.
\end{proof}
As the sum of two $s$-ward continuous function on $A$ is $s$-ward continuous and $cf$ is $s$-ward continuous for any constant real number $c$, the set of $s$-ward continuous functions on $A$ is a vector subspace of vector space of all continuous function on $A$.
\begin{theorem}\label{a}Every $s$-ward continuous function on $A\subset X$ is ward continuous on $A$.
\end{theorem}
\begin{proof} Assume that $(x_{k})$ is a quasi-Cauchy sequence in $A$ and $f$ is any $s$-ward continuous function on $A$. If $s=1$, the result is obvious. Let $s>1$ and a sequence $$(t_{m})=(\underbrace{x_{1}, x_{1},..., x_{1}}_{s-times}, \underbrace{x_{2}, x_{2},..., x_{2}}_{s-times},..., \underbrace{x_{n}, x_{n},..., x_{n}}_{s-times}, ...)$$ be $s$-quasi-Cauchy, i.e.
\begin{equation}
lim_{m\rightarrow\infty}||\Delta_s t_{m},\mu_1,\mu_2,...,\mu_{n-1}||=0.
\end{equation}
We have
\begin{equation}
lim_{m\rightarrow\infty}||\Delta_s f(t_{m}),f(\mu_1),f(\mu_2),...,f(\mu_{n-1})||=0
\end{equation}
by using the $s$-ward continuity of the function $f$.
Therefore
\begin{equation}
lim_{m\rightarrow\infty}||\Delta f(t_{m}),f(\mu_1),f(\mu_2),...,f(\mu_{n-1})||=0
\end{equation}
So $s$-ward continuity of the function $f$ implies that the ward continuity of $f$ on $A\subset X$.
\end{proof}
\begin{theorem} The image of an $s$-ward compact subset of $X$ by an $s$-ward continuous function is $s$-ward compact.
\end{theorem}
\begin{proof}
Assume that $f$ is an $s$-ward continuous function and $A$ is an $s$-ward compact subset of $X$. Choose a sequence $t$ as $t=(t_{k})\in f(A)$ and say $(t_{k})=f(x_{k})$ where $x_{k}\in {A}$. $A$ is $s$-ward compact so there is a subsequence $(x_{m})$ of $(x_k)$ with
\begin{equation}
\lim_{m\rightarrow \infty} ||\Delta_{s}x_{m},\mu_1,\mu_2,...,\mu_{n-1}||=0
\end{equation}
for all $\mu_1,\mu_2,...,\mu_{n-1}\in X$. Using the $s$-ward continuity of $f$ we have
\begin{equation}
\lim_{m\rightarrow \infty} ||\Delta_{s} f(x_{m}),f(\mu_1),f(\mu_2),...,f(\mu_{n-1})||=0,
\end{equation}
so there is an $s$-quasi-Cauchy subsequence $(f(x_{m}))$ of $t$. The result implies that the subset $f(A)\subset X$ is $s$-ward compact.
\end{proof}
$s$-ward continuous image of any compact subset of $X$ is compact. It is easily evaluated from Theorem \ref{s}.
\begin{theorem} If $f$ is uniformly continuous on $A\subset X$, then it is $s$-ward continuous on $A$.
\end{theorem}
\begin{proof}
Let $f$ be a uniformly continuous function on $A$, and the sequence $(x_{k})$ be an $s$-quasi-Cauchy sequence in $A$. Our aim is to prove the sequence $(f(x_{k}))$ is also an $s$-quasi-Cauchy sequence in $A$. Take any $\varepsilon>0$. There exists a $\delta>0$ such that if
\begin{equation}
||x-y,\mu_1,\mu_2,...,\mu_{n-1}||<\delta \ \ \textrm{then} \ \ ||f(x)-f(y),f(\mu_1),f(\mu_2),...,f(\mu_{n-1})||<\epsilon.
\end{equation}
There exists an $\tilde{k}=\tilde{k}(\delta)$ for this $\delta >0$ such that
\begin{equation}
||\Delta_{s} x_{k},\mu_1,\mu_2,...,\mu_{n-1}||<\delta
\end{equation}
for every $\mu_1,\mu_2,...,\mu_{n-1}\in X$ whenever $k>\tilde{k}$. Uniform continuity of $f$ on $A$ for every $k>\tilde{k}$ implies
\begin{equation}
||\Delta_{s} f(x_{k}),f(\mu_1),f(\mu_2),...,f(\mu_{n-1})||<\varepsilon
\end{equation}
for every $f(\mu_1),f(\mu_2),...,f(\mu_{n-1})\in X$. The sequence $(f(x_{k}))$ is $s$-quasi-Cauchy so the function $f$ is $s$-ward continuous.
\end{proof}

\begin{theorem} Uniform limit of a sequence of $s$-ward continuous function is $s$-ward continuous.
\end{theorem}
\begin{proof}Let $(f_{t})$ be a sequence of $s$-ward continuous functions and it be uniformly convergent sequence to a function $f$. Pick an $s$-quasi-Cauchy sequence $(x_{k})$ in $A$ and choose any $\varepsilon>0$. There is an integer $N\in Z^+$ such that
\begin{equation}
||f_{t}(x)-f(x),f(\mu_1),f(\mu_2),...,f(\mu_{n-1})||<\frac{\varepsilon}{3}
\end{equation}
for every $x \in {A}$, for all $f(\mu_1),f(\mu_2),...,f(\mu_{n-1})\in X$ whenever $t\geq N$. Using the $s$-ward continuity of $f_{N}$, there is a positive integer $N_1(\varepsilon)>N$ such that
\begin{equation}
||\Delta_s f_{t}(x_{k}),f(\mu_1),f(\mu_2),...,f(\mu_{n-1})||<\frac{\varepsilon}{3}
\end{equation}
for every $t\geq N_1$. Now for $t\geq N_1$  we have
\begin{eqnarray}
&||\Delta_s f(x_{k}),f(\mu_1),f(\mu_2),...,f(\mu_{n-1})||\nonumber\\
&=||f(x_{k+s})-f(x_k),f(\mu_1),f(\mu_2),...,f(\mu_{n-1})||\nonumber\\
&\leq||f(x_{k+s})-f_{N}(x_{k+s}),f(\mu_1),f(\mu_2),...,f(\mu_{n-1})||\nonumber\\&+||\Delta_s f_{N}(x_{k}),f(\mu_1),f(\mu_2),...,f(\mu_{n-1})||\nonumber\\&+||f_{N}(x_{k})-f(x_{k}),f(\mu_1),f(\mu_2),...,f(\mu_{n-1})|| <\frac{\varepsilon}{3} + \frac{\varepsilon}{3} + \frac{\varepsilon}{3}=\varepsilon.
\end{eqnarray}
So the function $f$ is $s$-ward continuous on $A$.
\end{proof}
\begin{theorem} The collection of the $s$-ward continuous functions on $A\subset X$ is a closed subset of the collection of every continuous functions on $A$.
\end{theorem}
\begin{proof}
Let $E$ be a collection of all $s$-ward continuous functions on $A\subset X$ and $\overline{E}$ is the closure of $E$. $\overline{E}$ is defined as for every $x\in X$ there exists $x_k\in E$ with $\lim_{k\rightarrow \infty}x_k=x$ and $E$ is closed if $E=\overline{E}$. It is obvious that $E\subseteq\overline{E}$. Let $f$ be any element of the set of all closure points of $E$ which means there exists a sequence of points $f_t$ in $E$ as
\begin{equation}
\lim_{t\rightarrow \infty} ||f_{t}-f,f(\mu_1),f(\mu_2),...,f(\mu_{n-1})||=0
\end{equation}
for all $f(\mu_1),f(\mu_2),...,f(\mu_{n-1})\in X$ and also $f_t$ is a $s$-ward continuous function. Choose the sequence $(x_{k})$ as any $s$-quasi-Cauchy sequence. Since $(f_{t})$
converges to $f$, for every $\varepsilon >0$ and $x \in {E}$, there is any $N_0$ such that for every $t \geq {N_0}$,
\begin{equation}
||f(x)-f_{t}(x),f(\mu_1),f(\mu_2),...,f(\mu_{n-1})||< \frac{\varepsilon}{3}.
\end{equation}
As $f_{N}$ is $p$-ward continuous, $N_{1}>N_0$ exists such that for all $t \geq {N_{1}}$,
\begin{equation}||\Delta_s f_{N}(x_{k}),f(\mu_1),f(\mu_2),...,f(\mu_{n-1})||<\frac{\varepsilon}{3}.
\end{equation}
Hence for all $t \geq {N_{1}}$,
\begin{eqnarray}
 &||\Delta_s f(x_{k}),f(\mu_1),f(\mu_2),...,f(\mu_{n-1})||\nonumber\\
 &=||f(x_{k+s})-f(x_k),f(\mu_1),f(\mu_2),...,f(\mu_{n-1})||\nonumber\\
 &\leq||f(x_{k+s})-f_{N}(x_{k+s}),f(\mu_1),f(\mu_2),...,f(\mu_{n-1})||\nonumber\\&+||f(x_{k})-f_{N}(x_{k}),f(\mu_1),f(\mu_2),...,f(\mu_{n-1})||\nonumber\\&+||\Delta_s f_{N}(x_{k}),f(\mu_1),f(\mu_2),...,f(\mu_{n-1})|| <\frac{\varepsilon}{3} + \frac{\varepsilon}{3} + \frac{\varepsilon}{3}= \varepsilon.
\end{eqnarray}
Since the function $f$ is $s$-ward continuous in $E$ then $E=\overline{E}$, using Theorem \ref{s}, it ends the proof.
\end{proof}

\section{Conclusion}
\label{Sec:3}
The notion of an $n$-normed space was given by thinking if there is a problem where $n$-norm topology works however norm topology doesn't. As an application of the notion of n-norm, we can examine that if a term in the definition of n-norm shows the change of shape then the n-norm stands for the associated volume of this surface. Suppose that for any particular output one needs n-inputs but with one main input and other (n-1)-inputs as dummy inputs required which accomplish the operation, so this concept may be used as an application in many areas of science. The generalization of the notions of quasi-Cauchy sequences and ward continuous functions to the notions of $s$-quasi-Cauchy sequences and $s$-ward continuous functions in $n$-normed spaces are investigated in this paper. Also we find out some interesting inclusion theorems related to the concepts of ordinary continuity, uniform continuity, $s$-ward continuity, and $s$-ward compactness. We prove that the uniform limit of a sequence of $s$-ward continuous function is $s$-ward continuous and the set of $s$-ward continuous functions is a closed subset of the set of continuous functions.
We recommend research $s$-quasi-Cauchy sequences of points and fuzzy functions in an $n$-normed fuzzy space as a further study. However, due to the different structure, the methods of proof will not be similar to the one in this study. (see \cite{kocinac}, \cite{fuzzy}). Also we recommend investigate $s$-quasi-Cauchy sequences of double sequences in $n$-normed spaces as another further study (see \cite{mursaleen}, \cite{khan}).

\section*{Declarations}
\textbf{Ethical Approval} \ Not applicable

\textbf{Competing interests} \ Not applicable

\textbf{Authors' contributions} \ Not applicable

\textbf{Funding} \ Not applicable

\textbf{Availability of data and materials} \ Not applicable

\end{document}